\documentclass[11pt,fullpage]{article}

\usepackage{amsmath}
\usepackage{graphicx}
\usepackage{amsfonts}
\usepackage{amssymb}
\usepackage[normalem]{ulem}
\usepackage{scrextend}
\usepackage{endnotes}
\usepackage{soul}
\usepackage[usenames]{color}
\numberwithin{equation}{section}
\usepackage{natbib}
 \bibpunct[, ]{(}{)}{,}{a}{}{,}%

\newtheorem{prop}{Proposition}
\newtheorem{theorem}{Theorem}

\newtheorem{corollary}{Corollary}
\newenvironment{proof}[1][Proof]{\textbf{#1:} }{\ \rule{0.5em}{0.5em}}

\renewcommand{\baselinestretch}{1.2}

\newcommand{\e}{\mathbb{E}}
\newcommand{\Z}{\mathbb{Z}}

\newcommand{\var}{\mathbb{V}\mbox{ar}}

\newcommand{\Pro}{\mathbb{P}}

\newcommand{\cL}{\mathcal{L}}
\newcommand{\cE}{\mathcal{E}}

\title{\LARGE     \vspace{0.5cm} Stability of a Queue Fed by Scheduled Traffic at Critical Loading 
}

\author{{\large   Victor F. Araman \hspace{0.5cm}  Peter W. Glynn\hspace{0.5cm}}
\thanks{The first author is with the  Olayan School of Business, American University of Beirut, Beirut, va03@aub.edu.lb.  The second author is with the Management Science and Engineering department at Stanford University, Stanford, CA, 74305, glynn@stanford.edu. }}

\begin{document}
    \maketitle \thispagestyle{empty}
\date{}
\renewcommand{\thefootnote}{\fnsymbol{footnote}}
\setcounter{footnote}{1}
\renewcommand{\baselinestretch}{1.2} {\small\normalsize}

\begin{abstract}    
Consider the workload process for a single server queue with deterministic service times in which customers arrive according to a scheduled traffic process. A scheduled arrival sequence is one in which customers are scheduled to arrive at constant interarrival times, but each customer’s actual arrival time is perturbed from her scheduled arrival time by a random perturbation. In this paper, we consider a critically loaded queue in which the service rate equals the arrival rate. Unlike a queue fed by renewal traffic, this queue can be stable even in the presence of critical loading. We identify a necessary and sufficient condition for stability when the perturbations have finite mean. Perhaps surprisingly, the criterion is not reversible, in the sense that such a queue can be stable for a scheduled traffic process in forward time, but unstable for the time-reversal of the same traffic process.

\end{abstract}

\maketitle

\section{Introduction}
In this paper we consider a single server queue fed by \textit{scheduled traffic}. In particular, the $n^{th}$ customer is scheduled to arrive at time $nh$ (with $h>0$), but her actual arrival time occurs at time $n\,h+\xi_n$. We will call the random variable (rv) $\xi_n$ the perturbation associated with customer $n h$ arrival time. 

Scheduled traffic naturally arises when modeling a number of real-world applications. Consider, for instance, a service facility that uses an appointment-based system. Driven by cost efficiency, such systems are becoming the norm in many industries, especially in labor intensive ones such as private banking, medical clinics, and even hair dressing salons. The system provider can often control the duration of the engagement (i.e. the processing time), while she still faces the uncertainty due to customers arriving early or late relative to their appointment times. Relative to renewal traffic, which is often poorly motivated as an arrival model, scheduled traffic seems better suited to many applied domains. 

Scheduled traffic has been analyzed in the literature since \cite{Winsten59}. Most of the work has been restricted to bounded perturbations. Early papers focused on the characterization of the waiting time distribution (see,  \cite{Winsten59}, \cite{Mercer60}, \cite{Loynes62}, and \cite{Mercer73}), which led to formulations that don't lend themselves to direct quantitative computation. \cite{Chen97} used scheduled traffic with deterministic service times to model aircraft landings and looked at the stability of the corresponding single server queue under bounded perturbations. 
\cite{Kingman62} obtained a heavy-traffic result for single server queues with general arrival processes and showed that the result applies to the case of scheduled traffic. Specifically, the heavy traffic limit theorem of \cite{Kingman62} for the equilibrium distribution of an $S/G/1$ queue ($``S"$ for scheduled traffic) is identical to that of the corresponding $D/G/1$ queue when $G$ has  finite variance and the perturbations are positive with finite mean. More recently, \cite{AramanGlynn12} proved an FCLT for scheduled traffic when the perturbations have infinite mean. The limit involves a fractional Brownian motion, from which they obtained a heavy traffic limit process for the workload. \cite{AramanGlynn22} establish  properties of scheduled traffic and show, for finite-mean Pareto-like perturbations, that an $S/D/1$ queue behaves very differently from both a $D/D/1$ and a $G/D/1$ queue. 

We assume throughout this paper that the $\xi_n$'s are independent and identically distributed (i.i.d.) rv's with $\e |\xi_0|<\infty$. To make the arrival point process time-stationary, we shift the time origin by a uniform amount $U$, so that the number of arrivals in $[a,b]$ is given by $$N(b)-N(a)=\sum_{j=-\infty}^\infty I(jh+hU+\xi_j\in[a,b])$$ for $a<b$, where $U$ is uniform on $[0,1]$, and independent of $\xi=(\xi_n:-\infty<n<\infty)$. We choose $N(0)=0$ in order to ``anchor" $N=(N(t):-\infty<t<\infty).$ 

Suppose that each customer $n$ has an associated service time $V_n$, where $V=(V_n:-\infty<n<\infty)$ is an i.i.d. sequence independent of $\xi$ and $U$. We assume that the server serves work at unit rate. If the work-in-system at $t=0$ is zero, then the workload in the system at time $t$ is given by 
$$W(t)=\max_{0\leq s\leq t}\sum_{i=N(s)}^{N(t)}V_i-(t-s).$$

If $\e V_1<h$, Loynes' Lemma guarantees that $W(t)\Rightarrow W(\infty)<\infty$ as $t\rightarrow\infty$, where $\Rightarrow$ denotes weak convergence and $W(\infty)$ is finite-valued, so that the queue is ``stable" (see, \cite{Loynes62}). On the other hand, if $\e V_1 > h$, $W(t)\Rightarrow\infty$ as $t\rightarrow\infty$, so that the system is ``unstable". 

If $\e V_1 = h$, the queue is said to be \textit{critically loaded}. We show in Corollary~1 that 
\begin{equation}t^{-1/2}\,W(t)\Rightarrow \sigma X(1)\label{Eq:Workload_rho=1}\end{equation} as $t\rightarrow\infty$, where $\sigma^2=\var ~V_1/h$, and $X=(X(s):s\geq 0)$ is standard reflected Brownian motion. It follows that if $\var ~V_1>0$, then  $W(t)\Rightarrow\infty$ as $t\rightarrow\infty$, so that the queue is unstable.  Note that Kingman's heavy traffic limit is for subcritical queues with a scaling involving $h-\e V_1$ (see \cite{Kingman62}), whereas our limit result (\ref{Eq:Workload_rho=1}) describes critical loading and normalizes by $t^{-1/2}$.

To this point in our discussion of stability, the stability characterization of a queue fed by scheduled traffic is identical to that of a queue fed by renewal traffic. However, if $\e V_1=h$ and $\var ~V_1=0$, so that the service times are deterministic, a new phenomenon appears. In particular, the $S/D/1$ queue  can be stable in the critically loaded regime. 
As mentioned above, \cite{Chen97} studied $S/D/1$ stability under such regime and showed that when the perturbations are bounded, the $S/D/1$ queue is stable. 
On the other hand, when the perturbations have a Pareto-like right tail, \cite{AramanGlynn22} show that the workload grows to infinity under critical loading. 

In view of the applied importance of scheduled traffic, this paper aims to give a necessary and sufficient condition for the stability of an $S/G/1$ queue when the perturbations have finite means; see Theorem~1. 

The theorem shows that $W(t)\Rightarrow W(\infty)$ as $t\rightarrow\infty$ with $W(\infty)$ being finite-valued if and only if $\xi_0^+\overset {\Delta}{=}\max (\xi_0,0)$ is a bounded rv. One interesting aspect of this stability characterization is that it is not symmetric under time-reversal. In particular, the perturbations of the time-reversal arrival process are i.i.d. and have a common distribution given by that of $-\xi_0$, so that the time-reversed is stable if $\xi_0^-\overset {\Delta}{=}\max (-\xi_0,0)$ is a bounded rv. So, stability is not preserved under time-reversal. Other asymmetric aspects of the behavior of the workload with respect to the arrival process and its time-reversal was also recently observed in \cite{AramanGlynn22}.
\section{The Main Result}
We first prove (\ref{Eq:Workload_rho=1}). For that we start by showing the following result.


\begin{prop}
If $\e|\xi_0|<\infty$, then 
$$\frac{1}{\log n}\max_{0\leq s\leq 1} \left |h\,N(ns)-ns\right |\Rightarrow 0$$
as $n\rightarrow \infty$.\label{Prop:M_cv_0}
\end{prop}
\vspace{0.cm}

\begin{proof} Define $\cE(t)=\sum_{ih+Uh>t}I(ih+Uh+\xi_i\leq t)$ and $\cL(t)=\sum_{ih+Uh<t}I(ih+Uh+\xi_i\geq t)$. If $\e |\xi_0|<\infty$, it is not hard to show that $\cE$ and $\cL$ are non-negative time-stationary sequences. Moreover, we can write  for $\theta>0$,
\begin{align*}
    \e\exp(\theta\cE(0))&=\e\exp(\theta\sum_{jh+Uh>0}I(jh+Uh+\xi_j\leq 0))\\
    &=\int_0^1\e\exp(\theta\sum_{j\geq 0}I(-\xi_j\geq -jh+uh))du\\
    &=\int_0^1 \exp(\sum_{j\geq 0}\log(1+\Pro(-\xi_0\geq -jh+uh)(e^{\theta}-1))du\\
    &\leq \int_0^1 \exp((e^{\theta}-1)\sum_{j\geq 0}\log(1+\Pro(-\xi_0\geq jh+uh))du\\
    &\leq \int_0^1 \exp((e^{\theta}-1)\sum_{j\geq 0}\Pro(-\xi_0/h-u\geq j))du\\
    &\leq\exp((e^\theta-1)(\e\xi^-/h+1))<\infty.
\end{align*}
A similar argument shows that  $$\e\exp(\theta\cL(0))\leq\exp((e^\theta-1)(\e\xi^+/h+1))<\infty.$$
From Proposition~2 in \cite{AramanGlynn22} we have that for any $t\geq 0$  
\begin{equation*}
N(t)=(\sum_{ih+Uh\in(0,t]}1)+(\cE(t)-\cL(t))-(\cE(0)-\cL(0)).    
\end{equation*}
Therefore, if $O_p(1)$ denotes a term stochastically bounded in $n$ we write,
\begin{align*}
    &\max_{0\leq s\leq nh}\left|\,N(s)-s/h\right|\\
    &\leq \max_{0\leq s\leq nh}\left[\cE(s)+\cL(s)\right]+O_p(1)\\
    &\leq \max_{0\leq k\leq n}\left[\cE((k+1)h)+\cL(kh)+\max_{0\leq s\leq h}\left[\cE((k+1)h-s)-\cE(k+1)h) \right]+\max_{0\leq s\leq h}\left[\cL(s+kh)-\cL(kh) \right]\right]+O_p(1)\\
    &\leq \max_{1\leq k\leq n+1}\cE(kh) +\max_{0\leq k\leq n}\cL(kh)+\max_{0\leq k\leq n+1}    [N((k+1)h)-N(kh)]+O_p(1).
\end{align*}

\vspace{0.2cm}

Also for $\varepsilon>0$ and $\theta>0$, 
\begin{align*}
\Pro(\max_{1\leq k\leq n+1}\cE(kh)>\varepsilon\,\log n) &\leq \sum_{k=1}^{n+1}\Pro(\cE(kh)>\varepsilon\,\log n)\\
     &=(n+1)\,\Pro(\cE(0)>\varepsilon\,\log n)\\
    &\leq (n+1)\, \e\frac{\exp(\theta\cE(0))}{e^{\theta\varepsilon\log n}}I(\exp(\theta\cE(0))>e^{\theta\varepsilon\log n}) \\
    &\leq (n+1)\,n^{-\theta\varepsilon}\,\e\exp(\theta\cE(0))
\end{align*}
By choosing $\theta\,\varepsilon>1$, we conclude that $\frac{1}{\log n}\max_{1\leq k\leq n+1}\cE(kh)\overset{p}{\rightarrow} 0$ as $n\rightarrow\infty$. Similarly, we show that $\frac{1}{\log n}\max_{0\leq k\leq n}\cL(kh)\overset{p}{\rightarrow} 0$ as $n\rightarrow\infty$. Finally, we use the fact that the rv's $\{N((k+1)h)-N(kh):0\leq k\leq n+1\}$ are identically distributed and $\e\exp(\theta N(h))<\infty$ for $\theta>0$ (see again \cite{AramanGlynn22}). Hence, the above argument proves similarly that $\frac{1}{\log n}\max_{0\leq k\leq n+1}(N((k+1)h)-N(kh))\overset{p}{\rightarrow} 0$ as $n\rightarrow\infty$,  proving the result. \end{proof}

\begin{corollary}
Suppose that $\e|\xi_0|<\infty$ and $\var ~V_1<\infty.$ If $h=\e V_1$, then $$t^{-1/2}\,W(t)\Rightarrow \sigma X(1)$$
as $t\rightarrow\infty$, where $X=(X(s):s\geq 0)$ is a reflected Brownian motion with mean 0 and unit volatility, with $X(0)=0.$
\end{corollary}

\begin{proof}
Note that 
\begin{align*}
    &\max_{0\leq s\leq t}\sum_{i=1}^{N(s)}V_i-s\\
    =&\max_{0\leq s\leq t}\sum_{i=1}^{N(s)}(V_i-\e V_i)+h\,N(s)-s
    \end{align*}
From Proposition~\ref{Prop:M_cv_0} we have that $\max_{0\leq s\leq t}h\,N(s)-s=o_p(\log t)$ where $o_p(a_t)$ is a quantity such that $\frac{1}{a_t}\overset{p}{\rightarrow}0$ as $t\rightarrow\infty$. Given that $N(ts)/t\rightarrow s/h~a.s.$ as $n\rightarrow\infty$, we conclude that
\begin{align*}
t^{-1/2}\max_{0\leq s\leq t}\sum_{i=1}^{N(s)}V_i-s    &=t^{-1/2}\max_{0\leq s\leq 1}\sum_{i=1}^{N(st)}(V_i-\e V_i)+ o_p(1)\\
&\Rightarrow \frac{\sigma}{\sqrt{h}} \,X(1)
\end{align*}
as $n\rightarrow\infty$
which proves the result. 
\end{proof}

\vspace{0.2cm}

We next turn to discuss the stability of the $S/D/1$ queue.

\begin{theorem}
Suppose that $(W(t):t\geq 0)$ is the workload process associated with the $S/D/1$ queue. Suppose that $V_1=h~a.s.$ and $W(0)=0.$ Then, there exists a finite-valued rv $W(\infty)$ such that $$W(t)\Rightarrow W(\infty)$$ as $t\rightarrow\infty$ if and only if $\xi_0^+$ is a bounded rv. 
\end{theorem}
\begin{proof}
Loynes' lemma ensures that $W(t)\overset{D}{=}M(t)$, where $(M(t):0\leq t\leq \infty))$ is the running maximum of the time-reversed arrival process, given by 
\begin{align*}
M(t)&=\max_{0\leq s\leq t}\left[\sum_{j=-\infty}^{\infty}I(jh+Uh-\xi_j\in(0,s])h-s\right]\\
&\overset{D}{=}\max_{0\leq s\leq t}\left[\Lambda(s)-s\right].
\end{align*}
Put $k(s)=\lfloor s/h\rfloor$ and note that
\begin{align*}
    \Lambda(s)-h\,\lfloor s/h\rfloor &= \sum_{j=1}^\infty I(Uh-jh-\xi_{-j}\in(0,s])\\
    &~~~~~~~~~+\sum_{j=k(s)+1}^\infty I(Uh+jh-\xi_{j}\in(0,s])\\
    &~~~~~~~~~-\sum_{j=1}^{k(s)} I(Uh+jh-\xi_{j}\in(-\infty,0])\\
    &~~~~~~~~~-\sum_{j=1}^\infty I(Uh+jh-\xi_{j}\in(s,\infty))\\
    &\overset{\Delta}{=}\beta_1(s)+\beta_2(s)-\beta_3(s)-\beta_4(s).
\end{align*}

Note that $$\beta_1(s)\leq \sum_{j=1}^\infty I(-\xi_{-j}>(j-1)\,h)\overset{\Delta}{=}\Gamma_1.$$
But $\e \Gamma_1<\infty$ if $\e\xi_0^-<\infty$, so that $\Gamma_1<\infty~a.s.$ Consequently, $$\max_{s\geq 0}\beta_1(s)<\infty~a.s.$$

Of course, $$M(t)\leq \max_{s\geq 0}\beta_1(s)+\max_{s\geq 0}\beta_2(s).$$

But if $\xi_0^+\leq c<\infty ~a.s.$, then 
\begin{align*}
    \beta_2(s)\leq \sum_{j=k(s)+1}^\infty I(-\xi_j\leq s-jh)&\leq\sum_{j=1}^\infty I(\xi_{k(s)+j}\geq jh)\\
    &\leq (\lfloor c\rfloor+1)/h
\end{align*}
and hence $M(\infty)<\infty~~a.s.$

Suppose now that $\xi_0^+$ has infinite support. Observe that $$\beta_4(s)\leq \sum_{j=1}^\infty I(\xi_{k(s)-j}>(j-2)\,h)\overset{\Delta}{=}Y_{k(s).}$$

Since $\e \xi_0^-<\infty$, $Y_0<\infty~a.s.$ and there exists $l<\infty$ such that $$\Pro(Y_0>l)=\Pro(Y_j>l)<1.$$

The rv $Y_n$ can be represented as $Y_n=f(\xi_{n+j}:j\in\mathbb{Z})$, where the $\xi_j$'s are i.i.d.. The ergodic theorem therefore ensures that 
$$\frac{1}{n}\sum_{j=1}^{n}I(Y_j\leq l)\rightarrow \Pro(Y_0\leq l)>0~~a.s.$$
as $n\rightarrow\infty,$ so that there exists a non-negative sequence $(\tau_i:i\geq 1)$ such that $\tau_i\rightarrow\infty$ as $i\rightarrow\infty$ and $Y_{\tau_i}\leq l$ for $i\geq 1$. Note that the $\tau_i$'s are stopping times adapted to $({\cal F}_j:j\geq 0)$ , where ${\cal F}_j=\sigma(U,\xi_k:-\infty<k\leq j)$. 

Observe that $$\beta_2(s)\geq \sum_{j=1}^{\infty}I(\xi^+_{k(s)+j}\geq (j+2)\,h)\overset{\Delta}{=}\tilde Y_{k(s)}.$$ For each $r\in\Z_+$, the infinite support of $\xi_0^+$ implies that $$\Pro(\tilde Y_0>r)=\Pro(\tilde Y_i>r)>0.$$ Because $\tilde Y_{\tau_i}$ is independent of $(\xi_j:j\leq \tau_i),$ 
\begin{align*}
\Pro(\beta_2(\tau_i h)>r~|{\cal F}_{\tau_i})&\geq \Pro(\tilde Y_{k(\tau_i h)}>r~|{\cal F}_{\tau_i})\\
&=\Pro(\tilde Y_{\tau_i}>r~|{\cal F}_{\tau_i})\\
&=\Pro(\tilde Y_{0}>r).
\end{align*}

It follows from the conditional Borel Cantelli lemma (see, \cite{Doob}) that $\beta_2(\tau_i h)>r$ infinitely often a.s. Finally, 
\begin{align*}
\beta_3(s)&\leq \sum_{j=1}^\infty I(Uh+jh-\xi_j\in(-\infty,0])\\
&\leq \sum_{j=1}^\infty I(\xi_j\geq (j-1)\,h)\overset{\Delta}{=}\Gamma_2<\infty~a.s.
\end{align*}
since $\e\xi_0^+<\infty.$ Hence,
\begin{align*}
    \max_{s\geq 0} & ~\Lambda(s)-s\\
    &\geq \limsup_{i\rightarrow\infty}\beta_2(\tau_i h)-\Gamma_2-Y_{\tau_i} \\
    &\geq r- \Gamma_2- l~~a.s.
\end{align*}
Since $r$ can be made arbitrarily large, we may conclude that $$\max_{s\geq 0}\Lambda(s) - s=\infty ~~a.s.,$$ so that $W(t)\Rightarrow\infty$ as $t\rightarrow\infty,$ proving the theorem.
\end{proof}
\vspace{0.2cm}

According to Theorem~1, the $S/D/1$ queue can be stable in critical loading. Furthermore, Theorem~1 establishes that only the right tail of $\xi_0$ affects stability. The fact that the right tail of $\xi_0$ appears to have a greater impact on the performance of the $S/D/1$ queue than does the left tail can also be seen in the heavy traffic limit theory for the $S/D/1$ queue; see Theorems~5 and 6 of \cite{AramanGlynn22} for details. Intuitively, ``late arrivals" (controlled by the the right tail of $\xi_0$) affect the performance of a scheduled queue in a greater degree than the ``early arrivals" (controlled by the left tail of $\xi_0$).  
\bibliographystyle{ormsv080}
\bibliography{BiblioApP}
\end{document}